\documentclass{amsart}

\title[Topological $G$-manifolds II]{Countable approximation of topological $G$-manifolds, II:  linear Lie groups $G$}

\author[Q Khan]{Qayum Khan}
\address{Department of Mathematics \hfill Indiana University \hfill Bloomington IN 47405 USA} 
\email{qkhan@indiana.edu}

\usepackage{amsthm,amsmath,amssymb,eucal,mathrsfs}
\usepackage{amscd,pb-diagram,verbatim}
\usepackage[all]{xy}

\usepackage{url}

\usepackage{xcolor}
\definecolor{dark-red}{rgb}{0.4,0.15,0.15}
\definecolor{dark-blue}{rgb}{0.15,0.15,0.4}
\definecolor{medium-blue}{rgb}{0,0,0.5}
\usepackage[colorlinks, linkcolor={dark-red}, citecolor={dark-blue}, urlcolor={medium-blue}]{hyperref} 

\newtheorem{thm}{Theorem}[section]

\newtheorem{cor}[thm]{Corollary}
\newtheorem{lem}[thm]{Lemma}

\theoremstyle{definition}
\newtheorem{defn}[thm]{Definition}

\newtheorem{rem}[thm]{Remark}
\newtheorem{exm}[thm]{Example}

\numberwithin{equation}{section}

\DeclareMathAlphabet{\matheurm}{U}{eur}{m}{n}

\newcommand{\C}{\mathbb{C}}

\newcommand{\N}{\mathbb{N}}
\newcommand{\Q}{\mathbb{Q}}
\newcommand{\R}{\mathbb{R}}

\newcommand{\Z}{\mathbb{Z}}

\newcommand{\cA}{\mathcal{A}}
\newcommand{\cC}{\mathcal{C}}

\newcommand{\cI}{\mathcal{I}}
\newcommand{\cM}{\mathcal{M}}

\newcommand{\cP}{\mathcal{P}}

\newcommand{\cT}{\mathcal{T}}
\newcommand{\cU}{\mathcal{U}}

\newcommand{\Cpt}{\mathrm{Cpt}}

\newcommand{\Homeo}{\mathrm{Homeo}}

\newcommand{\id}{\mathrm{id}}

\newcommand{\Isom}{\mathrm{Isom}}

\newcommand{\Map}{\mathrm{Map}}

\newcommand{\Torus}{\mathrm{Torus}}

\newcommand{\eps}{\varepsilon}

\newcommand{\iso}{\cong}
\newcommand{\longra}{\longrightarrow}

\newcommand{\x}{\times}

\newcommand{\ol}[1]{\overline{#1}}

\begin{document}

\begin{abstract}
Let $G$ be a matrix group.
Topological $G$-manifolds with Palais-proper action have the $G$-homotopy type of countable $G$-CW complexes (\ref{cor:main}).
This generalizes E.~Elfving's dissertation theorem for locally linear $G$-manifolds (1996).
Also we improve the Bredon--Floyd theorem from compact groups $G$.
\end{abstract}
\maketitle

\section{Equivariant cohomology manifolds}

\begin{defn}
Let $G$ be a topological group.
Let $X$ be a \textbf{$G$-space}, that is, a topological space equipped with a left $G$-action.
For any $x \in X$, its \textbf{orbit type} $(G_x)$ is the $G$-conjugacy class of its \textbf{isotropy group} $G_x := \{g \in G ~|~ gx=x\} \leqslant G$.
The $G$-space $X$ is \textbf{$G$-metrizable} if it has a $G$-invariant metric: $d(gx,gy)=d(x,y)$.
The group $G$ is \textbf{Lie} if $G$ is a real-analytic manifold with $(a,b) \longmapsto a^{-1} b$ analytic.
\end{defn}

Proper in the sense of Bourbaki is \cite[I:3.17]{tomDieck} and of Palais is \cite[1.2.2]{Palais}.

\begin{defn}
Let $G$ be a topological group.
Let $X$ be a $G$-space.
Define a map
\[
\theta: G \x X \longra X \x X \quad;\quad (g,x) \longra (x,gx).
\]
The $G$-space $X$ is \textbf{Bourbaki-proper} if $\theta$ is proper in the sense of Bourbaki: the product $\theta \x \id_Z$ is a closed function for any topological space $Z$ \cite[10:1.1]{Bourbaki}.

Now suppose $G$ is locally compact and $X$ is completely regular, both without any assumption of Hausdorff.
The $G$-space $X$ is \textbf{Palais(-proper)} if each point $x \in X$ has a neighborhood $U$ such that any $y \in X$ has a neighborhood $V$ with
\[
\langle U, V \rangle_G := \{g \in G ~|~ U \cap gV \neq \varnothing\} 
\]
precompact \cite[1.2.2]{Palais}.
More generally, $X$ is \textbf{Cartan(-proper)} if each point $x \in X$ has a neighborhood $U$ such that $\langle U, U \rangle_G \subset G$ is precompact \cite[1.1.2]{Palais}.
\end{defn}

We recall $\Z$-cohomology manifold \cite[I:3.3]{Borel_seminar}, without separable or metrizable.
Our description uses \v{C}ech cohomology for Alexander--Spanier cohomology~\cite{Dowker4}.

\begin{defn}[Borel]
Let $M$ be a locally compact, Hausdorff topological space.
Let $n \in \N$.
Then $M$ is an \textbf{$n$-dimensional $\Z$-cohomology manifold} ($n\text{-cm}_\Z$) if:
\begin{enumerate}
\item $\dim_\Z(M) \leqslant n$, that is, $\check{H}^n(M;\Z) \longra \check{H}^n(A;\Z)$ is onto for all closed $A \subset M$

\item $\forall x \in M$: local Betti numbers $\beta^i_\Z(M,x)=0 \; \forall i<n$ and $\beta^n_\Z(M,x)=1$ in the sense of Borel~\cite[I:2.1]{Borel_seminar} extending Aleksandrov (1935) and \v{C}ech (1934)

\item $\forall x \in M$: a local $\Z$-orientation of $M$ at $x$ exists in the sense of \cite[I:3.2]{Borel_seminar}.
\end{enumerate}
\end{defn}

We generalize the Bredon--Floyd theorem \cite[VII:2.2]{Borel_seminar} to noncompact groups, by adapting the circle of ideas within Floyd's initial argument for \cite[VI:1.1]{Borel_seminar}.

\begin{thm}\label{thm:BredonFloyd}
Let $G$ be a Lie group.
Let $M$ be a $\Z$-cohomology manifold with Bourbaki-proper $G$-action.
Any compact set in $M$ has only finitely many orbit~types.
\end{thm}

\begin{proof}
Assume not.
Then there exists an infinite sequence $\{x_i\}_{i=0}^\infty$ in some compact subset $K$ of $M$ such that no two of the isotropy groups $G_{x_i}$ are conjugate in $G$.
Since the action is proper, $C := \{ g \in G ~|~ gK \cap K \neq \varnothing \}$ is compact \cite[I:3.21]{tomDieck}.
In particular, since each $G_{x_i}$ is a closed subset of $C$, each $G_{x_i}$ is compact.
Recall that the set $\Cpt(G)$ of nonempty compact subsets of a metric space $(G,d)$ admits the Hausdorff--Pompeiu metric $d_{HP}$, which is compact if the ambient metric space is compact \cite[45:7]{Munkres}.
Then the infinite sequence $\{G_{x_i}\}_{i=0}^\infty$ in the compact metric space $(\Cpt(C),d_{HP})$ has a convergent subsequence, which we may reindex to be the original.
By continuity of multiplication and inversion in $G$, the compact subset $H := \lim_{i \to \infty} G_{x_i}$ is a subgroup of $G$.
Thus $H$ is a Lie group \cite[20.10]{Lee}.

Let $U$ be a compact neighborhood of the neutral element in $G$.
On the one hand, since $H$ is a compact Lie group acting on the $\Z$-cohomology manifold $M$, by the Bredon--Floyd theorem \cite[VII:2.2]{Borel_seminar}, the compact set $UK \subset M$ supports only finitely many $H$-orbit types.
On the other hand, by the Montgomery--Zippin neighboring-subgroups theorem \cite[4.2]{Palais}, there is a neighborhood $N$ of $H$ in $G$ so any subgroup of $G$ contained in $N$ is $U$-conjugate to a subgroup of $H$.
Since $H$ is a limit, there exists $i_0$ such that $G_{x_i} \subset N$ for all $i \geqslant i_0$.
Re-index so that $i_0=0$.
Then there exists $u_i \in U$ such that $G_{u_i x_i} = u_i G_{x_i} u_i^{-1} \subset H$ for each $i$.
Note $\{u_i x_i\}_{i=0}^\infty$ is an infinite sequence in $U K$ such that no two $G_{u_i x_i}$ are $G$-conjugate hence not $H$-conjugate, contradicting that $UK$ has only finitely many $H$-orbit types.
\end{proof}

\section{Equivariant absolute neighborhood retracts}

Recall $X$ is a \textbf{$G$-ANR for the class $\cC$} ($\cC$-absolute $G$-neighborhood retract) if $X$ belongs to $\cC$ and, for any closed $G$-embedding of $X$ into a member of $\cC$, there is a $G$-neighborhood of $X$ with $G$-retraction to $X$.
More generally, $X$ is a \textbf{$G$-ANE for the class $\cC$} ($\cC$-absolute $G$-neighborhood extensor) if, for any member $B$ of $\cC$ and closed $G$-subset $A$ of $B$ and any $G$-map $A \longra X$, there is a $G$-extension $U \longra X$ from some $G$-neighborhood $U$ of $A$ in $B$.
Notice a $G$-ANE need not belong to $\cC$.

Not long ago, S.~Antonyan \cite[5.7]{Antonyan2} made equivariant O.~Hanner's open-union theorem (see \cite[III:8.3]{Hu}), providing a local-to-global principle for $G$-extensors.

\begin{thm}[Antonyan]\label{thm:Antonyan}
Let $G$ be a locally compact Hausdorff group.
Let $\cC$ be a subclass of the class $G\text{-}\cP$ of paracompact Palais $G$-spaces with paracompact orbit space.
Any union of open $G$-subsets that are $G$-ANEs for $\cC$ is also a $G$-ANE for $\cC$.
\end{thm}

Equivariant CW structures were found over very general groups, using the nerves of locally finite coverings of neighborhoods in certain $G$-Banach spaces \cite[1.1]{AE2}.
Recall that T.~Matumoto defined the notion of a $G$-CW complex \cite[1.2, 1.5]{Matumoto_GCW}.

\begin{thm}[Antonyan--Elfving]\label{thm:AE}
Let $G$ be a locally compact Hausdorff group.
Suppose that $X$ is a $G$-ANR for the class $G\text{-}\cM$ of $G$-metrizable Palais $G$-spaces.
Then $X$ has the equivariant homotopy type of a $G$-CW complex with Palais action.
\end{thm}

\begin{rem}\label{rem:RudinStone}
Observe that the class $G\text{-}\cM$ is a subclass of $G\text{-}\cP$, as follows.
Let $X$ be a member of $G\text{-}\cM$.
Since $X$ is $G$-metrizable, the orbit space $X/G$ has an induced metric given by an infimum.
Then, since both $X$ and $X/G$ are metrizable, by Stone's theorem \cite[41.4]{Munkres}, we have that both $X$ and $X/G$ are paracompact.
\end{rem}

As classes, observe $\cC \cap G\text{-ANE}(\cC) \subseteq G\text{-ANR}(\cC)$; a converse is \cite[6.3]{AAM}.

\begin{thm}[Antonyan--Antonyan--Mart\'in-Peinador]\label{thm:AAM}
Let $G$ be a locally compact Hausdorff group.
Then $G\text{-ANR}(G\text{-}\cM) ~=~ G\text{-}\cM ~\cap~ G\text{-ANE}(G\text{-}\cM)$.
\end{thm}

The following technical notion over compact groups was introduced in \cite{Jaworowski2}.
We restate from \cite[2.2]{AAMV} the generalization over noncompact groups.

\begin{defn}[Jaworowski]
Let $G$ be a Lie group.
A Palais $G$-space $X$ has \textbf{finite structure} if it has only finitely many orbit types and, for each orbit type $(H)$, the quotient map $X_{(H)} \longra X_{(H)}/G$ is a $G/H$-bundle with only finitely many local trivializations.
Here $(H)$ is the conjugacy class of $H$ in $G$, $X_{(H)} := \{x \in X ~|~ (G_x)=(H) \}$ is the $(H)$-stratum, $G_x := \{g \in G ~|~ gx=x\}$ is an isotropy group.
\end{defn}

\begin{rem}
Any compact Lie group is \textbf{linear}: it has an isomorphic topological embedding into $GL_n(\R)$ for some $n$.
This is a case of the following consequence of the Peter--Weyl theorem: any compact topological group $G$ embeds into a product of unitary groups; if $G$ has no small subgroups this product is finite \cite[4.1]{Khan}.
\end{rem}

Recall $X^H := \{x \in X ~|~ \forall g \in H: gx=x\}$ denotes the \textbf{$H$-fixed subspace of $X$}.
In the following recent theorem \cite[6.1]{AAMV}, the Jaworowski--Lashof criterion for $G$-ANRs \cite{Jaworowski2} is generalized from compact Lie groups $G$ to all linear ones.

\begin{thm}[Antonyan--Antonyan--Mata-Romero--Vargas-Betancourt]\label{thm:AAMV}
Let $G$ be a linear Lie group.
Let $X$ be a $G$-metrizable Palais $G$-space with finite structure.
Then $X$ is a $G$-ANR for the class of $G$-metrizable Palais $G$-spaces, if and only if $X^H$ is an ANR for the class of metrizable spaces for each compact subgroup $H<G$.
\end{thm}

\section{Equivariant topological manifolds}

\begin{thm}\label{thm:main}
Let $G$ be a linear Lie group.
Let $M$ be a cohomology manifold over $\Z$ that is both separable and metrizable.
Suppose $M$ has Palais $G$-action and the fixed set $M^H$ is ANR for the class of metrizable spaces for each compact subgroup $H$ of $G$.
Then $M$ is $G$-homotopy equivalent to a countable proper $G$-CW complex.
\end{thm}

\begin{proof}
Let $M$ be a $\Z$-cohomology manifold.
Since $M$ is separable and locally compact, there exists an increasing infinite sequence $\{M_i\}_{i=0}^\infty$ of open sets in $M$ whose union is $M$ and whose closures $\ol{M_i}$ in $M$ are compact.
By Theorem~\ref{thm:BredonFloyd}, the compact set $\ol{M_i}$, hence $M_i$, has only finitely many conjugacy classes of isotropy group.
The $G$-saturation $G M_i = \bigcup_{g \in G} g M_i$ is also open \cite[I:3.1(i)]{tomDieck} and has only finitely many $G$-orbit types.
Since $(G M_i)^H = G M_i \cap M^H$ is open in the ANR $M^H$, we have that $(G M_i)^H$ is also an ANR by Hanner's global-to-local principle \cite[III:7.9]{Hu}.

Since $G$ is a Lie group and $G \ol{M_i}$ is a Palais $G$-space, by Palais' slice theorem \cite[2.3.1, 2.1.2]{Palais}, $G \ol{M_i}$ has a covering $\cT_i$ by $G$-tubes of varying orbit types.
Furthermore, since $(G\ol{M_i})/G = \ol{M_i}/G$ is compact, $\cT_i$ can be assumed finite.
The stratum $(G M_i)_{(H)}$ of $G M_i \subset G \ol{M_i}$ has a single orbit type, so restriction of $\cT_i$ to it gives a finite covering by local trivializations of a $G/H$-fiber bundle with structure group $G$.
So the Palais $G$-space $G M_i$ has finite structure.
By Palais' metrization theorem \cite[4.3.4]{Palais}, the separable metrizable $M$, hence $G M_i$, is $G$-metrizable.
Since $G$ is linear, $G M_i$ is a $G$-ANR for $G\text{-}\cM$ (\ref{thm:AAMV}), hence is a $G$-ANE for $G\text{-}\cM$ (\ref{thm:AAM}).


Thus, by Remark~\ref{rem:RudinStone} and Theorem~\ref{thm:Antonyan}, $M = \bigcup_{i \in I} G M_i$ is a $G$-ANE for $G\text{-}\cM$.
Then, since $M$ is also member of $G\text{-}\cM$, $M$ is a $G$-ANR for $G\text{-}\cM$.
Therefore, by Theorem~\ref{thm:AE}, we conclude $M$ has the $G$-homotopy type of a proper $G$-CW complex.

We now make some remarks on how to guarantee only countably many $G$-cells.
The proof of Theorem~\ref{thm:AE} starts in \cite[5.2]{AAR}, with a closed $G$-embedding of $X$ into a $G$-normed linear space $L$ with Palais action on some $G$-neighborhood.
Specifically, those authors take $L=E \times N$ \cite[3.10]{AAR}, which is valid for any $G$-metrizable Palais $G$-space $X$.
Since our $X=M$ is locally compact, \emph{alternatively use} the simpler and more classical $G$-Banach space $L=C_0(X)$, where
\begin{eqnarray*}
C_0(X) &:=& \{ f \in C(X) ~|~ \forall \eps>0, \exists \text{ compact } K \subset X, \forall x \in X-K : |f(x)|<\eps \}\\
\|f\| &:=& \sup \{ |f(x)| ~|~ x \in X \}, \text{which is well-defined}.
\end{eqnarray*}
Indeed, E.~Elfving in \cite[Propositions~2,3]{Elfving2} showed the existence of a Kurotowski-like $G$-embedding of $X$ into $C_0(X)-\{0\}$ on which the continuous $G$-action is Palais.

Since $X$ is separable, there exists a countable dense subset $\Delta \subset X$.
Since $X$ is locally compact, the Alexandroff one-point compactification $X^*$ exists.
Since $X$ is second-countable, so is $X^*$, hence $X^*$ admits a metric $d$ by the Urysohn metrization theorem \cite[34.1]{Munkres}.
Consider the countable collection $\Delta_d \subset C(X^*)$ defined by
\[
\Delta_d ~:=~ \{1\} ~\cup~ \{d(-,p) \in C(X^*) ~|~ p \in \Delta\}.
\]
Since $\Delta_d$ contains a nonzero constant function and separates points because $\Delta$ is dense in $X^*$, by the Stone--Weierstrass theorem \cite[Corollary~3, p174]{Stone2}, the countable subring $\Q\langle\Delta_d\rangle$, consisting of rational polynomials in the elements of $\Delta_d$:
\[
\Q\langle\Delta_d\rangle ~:=~ \mathrm{Im}\left( \Q[\Delta_d] \longra C(X^*) \right)
\]
is dense in $C(X^*)$.
Hence $C_0(X) \subset C(X^*)$ is separable.

Then the $G$-neighborhood $U$ of $X$ in $C_0(X)-\{0\}$, on which the $G$-retraction $U \longra X$ is defined, is Lindel\"of, as it is separable and metrizable.
So in the proof of \cite[Proposition~5.2]{AE2}, the rich $G$-normal cover $\cU$ with index set $G \times \cM$ can be assumed to have $\cM$ a countable set.
The geometric $G$-nerve $K(\cU)$ is indexed \cite[p166]{AE2} by certain finite subsets of $\cM$.
Thus the semisimplicial $G$-space $K(\cU)$ has only countably many $G$-cells, according to the proof of \cite[Theorem~5.3]{AE2}, which relies on S.~Illman \cite{Illman_Lie} and this in turn involves only countably many $G$-cells for a smooth $G$-manifold.
Finally, since \cite[Proposition~5.2]{AE2} states that $K(\cU)$ $G$-dominates $X$, by a $G$-version of Mather's trick (see second paragraph of \cite[Proof~2.5]{Khan}), the $G$-CW complex for $X=M$ has only countably many $G$-cells.

For the convenience of the reader, we detail the conclusion of this last sentence.
Since the $G$-CW complex $K(\cU)$ $G$-dominates $X$, there are $G$-maps $u: X \longra K(\cU)$, $d: K(\cU) \longra X$, and $G$-homotopy $h: X \times [0,1] \longra X$ from $h_1 = d \circ u$ to $h_0 = \id_X$.
By $G$-cellular approximation, there exists a cellular $G$-map $\alpha: K(\cU) \longra K(\cU)$ that is $G$-homotopic to $u \circ d$ \cite[II:2.1]{tomDieck}.
On the one hand, the mapping torus
\[
\Torus(\alpha) ~:=~ \frac{K(\cU) \x [0,1]}{(x,1) \sim (\alpha(x),0)}
\]
is a $G$-CW complex \cite[I:1.11]{tomDieck}.
On the other hand it is $G$-homotopy equivalent~to
\[
\Torus(u \circ d) ~\simeq_G~ \Torus(d \circ u) ~\simeq_G~ \Torus(\id_X) ~=~ X \x S^1.
\]
Thus $X \simeq X \times \R$ is $G$-homotopy equivalent to the infinite cyclic cover of $\Torus(\alpha)$, namely the bi-infinite mapping telescope of $\alpha$ | a countable $G$-CW complex.
\end{proof}

Finally, we generalize \cite[2.5]{Khan} from $G$ being compact.
Note that the manifold must be noncompact if $G$ is noncompact in order for the action to be Cartan-proper.

\begin{cor}\label{cor:main}
Let $G$ be a linear Lie group.
Any topological $G$-manifold with Palais action has the equivariant homotopy type of a countable proper $G$-CW complex.
\end{cor}

Here, by \textbf{topological $G$-manifold} \cite[2.2]{Khan}, we mean the $H$-fixed subspace is a topological ($C^0$) manifold for each closed subgroup $H$ of a topological group $G$.
Herein, a topological manifold shall be separable, metrizable, and locally euclidean.

\begin{proof}
Let $M$ be a topological $G$-manifold with Palais action.
By Hanner's local-to-global principle \cite[III:8.3]{Hu}, each manifold $M^H$ is an ANR for the class of metrizable spaces.
Also $M$ is separable, metrizable, and a $\Z$-cohomology manifold.
Therefore we are done by Theorem~\ref{thm:main}.
\end{proof}

Thus more tractible are its Davis--L\"uck $G$-spectral homology groups \cite[3.7, 4.3]{DL1}, since we conclude countability of the $G$-CW complex that left-approximates.

\begin{cor}
Let $G$ be a linear Lie group.
Let $f: M \longra N$ be a $G$-map between topological $G$-manifolds with Palais actions.
Then $f$ is a $G$-homotopy equivalence if and only if $f^H: M^H \longra N^H$ is a homotopy equivalence for each closed $H$ of $G$.
\end{cor}

\begin{proof}
This is immediate from Corollary~\ref{cor:main} and the corresponding theorem for $G$-CW complexes \cite[II:2.7]{tomDieck}, which is proven using $G$-obstruction theory.
\end{proof}

In particular, we generalize the main result of Elfving's thesis \cite[4.20]{Elfving1}.
The definition of \emph{locally linear}, along with some discussion, is found in \cite[3.6, 3.7]{Khan}.
Note any smoothable action is locally linear, but not vice versa; see \cite[VI:9.6]{Bredon_TG}.

\begin{cor}[Elfving]
Let $G$ be a linear Lie group.
Let $M$ be a locally linear $G$-manifold with Palais action.
If $M$ has only finitely many orbit types, then $M$ has the equivariant homotopy type of a $G$-CW complex.
\end{cor}

\begin{proof}
This special case now follows immediately from Corollary~\ref{cor:main}.
\end{proof}

\section{Examples that are not locally linear}

We continue the three families of uncountable examples of \cite[3.1, 3.2, 3.3]{Khan}.
The purpose here is to show there do exist topological $G$-manifolds that are not locally linear when $G$ is a noncompact linear Lie group with torsion.
(All principal bundles are trivial if $G$ is connected torsionfree, such as $G=\R$ for complete flows.)

Their common trick is that the diagonal action will become Palais \cite[1.3.3]{Palais}, even though it is not on the first factor, using a homogenous space $G/H$ with $H$ compact for the second factor.
These $G/H$ are exactly those with transitive Palais $G$-action.
The transitivity on the second factor guarantees the same quotient space as the first's.
Any $C^1$ Palais action by a Lie group is $C^\omega$ \cite{Illman_smoothing, Illman_verystrong}; ours are $C^0$.

Indeed there is no contradiction to Palais' slice theorem \cite[2.3.1, 2.1.2]{Palais}.
There does exist a $G_x$-slice for each point $x$ of the Palais $G$-manifolds, but \emph{not all the slices are euclidean}, and this is why in particular these slices are not $G_x$-linear.

\begin{exm}[Bing]\label{exm:Bing}
Consider the double $D := E \cup_A E$ of the non-simply connected side $E$ in $S^3$ of the Alexander horned sphere $A \approx S^2$, whose embedding is \emph{not locally flat}.
This double has obvious involution $r_B$ that interchanges the two pieces and leaves the horned sphere fixed pointwise.
Bing showed $D$ is homeomorphic to $S^3$ \cite{Bing}.
Thus $r_B$ minus a fixed point (so on $\R^3$) negatively answers a question of Montgomery \cite[39b]{Eilenberg}, asking if the action is conjugate to an isometric one.

Consider the Lie group $G = \Isom(\R) = \R \rtimes_{-1} O_1$, a closed subgroup of $GL_2(\R)$:
\[
\left\langle \begin{pmatrix}e^t & 0\\ 0 & e^{-t}\end{pmatrix}, \begin{pmatrix}0 & 1\\ 1 & 0 \end{pmatrix} ~\bigg|~ t \in \R \right\rangle.
\]
Define a non-Cartan action of $G$ on $S^3$ by epimorphism to $O_1 \iso \langle r_B \rangle \leqslant \Homeo(S^3)$.
As noted above, the diagonal action of $G$ on the product of $S^3$ and the homogenous space $\R = G/O_1$ is Palais.
Then Corollary~\ref{cor:main} applies to the topological $G$-manifold $S^3 \x \R$.
In the orbit space $(S^3 \x \R)/G = S^3/r_B = E$, the stratum $A$ is not locally cofibrant, so the $C^0$ action of $G$ on the 4-manifold $S^3 \x \R$ cannot be locally linear.

For each $n \geqslant 3$, Lininger \cite[9, 10]{Lininger2} applies \cite{Bing2} to produce uncountably many inequivalent involutions on $S^n$ with fixed set an $(n-1)$-sphere and quotient not a manifold-with-boundary, so none is equivalent to a locally linear action.
They arise from uncountably inequivalent embeddings in $S^{n-1}$ of Cantor's space $2^{\N}$; in the form of multiparameter Antoine necklaces, the $n=4$ case is due to Sher \cite{Sher}.
\end{exm}

\begin{exm}[Montgomery--Zippin]
Adaptation of Bing's 1952 idea produces an involution $r_{MZ}$ of $S^3$ whose fixed set is an \emph{embedded circle} $K$ that is not locally flat \cite[\S2]{MZ}.
In Example~\ref{exm:Bing}, replacing $r_B$ and $A$ with $r_{MZ}$ and $K$ works verbatim.
Note $r_{MZ}$ preserves orientation and was first to negatively answer the $C^0$ version of a question of Paul~A~Smith \cite[36]{Eilenberg}, asking if the fixed circle is unknotted.

Alford gave uncountably many inequivalent involutions fixing a wild circle \cite{Alford}.

Higher \emph{codimension-two} examples are provided by Lininger.  He uses rotation of the Alexander horned sphere $A$ in 4-space to obtain a semifree $U_1$-action on $S^4$ with fixed set a 2-sphere \cite[7]{Lininger2}.
More generally, using Bing's later techniques \cite{Bing2}, he obtains uncountably many inequivalent semifree $U_1$-actions on $S^n$ whose fixed set is an $(n-2)$-sphere and quotient not a manifold-with-boundary \cite[8, 10]{Lininger2}.
\end{exm}

\begin{exm}[Lininger]
For each $k \geqslant 3$, there are uncountably many inequivalent \emph{free} $U_1$-actions on $S^{2k-1}$ whose quotients are not $C^0$ manifolds \cite[Remark 2]{Lininger}.
At the root of Lininger's work are Andrews--Curtis decomposition spaces \cite{AC}:
non-euclidean quotients $Q$ by a wild arc, any of whose product with $\R$ is euclidean.

Consider the Lie group $G=\Isom^+(\C)=\C\rtimes U_1$, a closed subgroup of $GL_2(\C)$:
\[
\left\langle \begin{pmatrix}1 & c\\ 0 & 1\end{pmatrix}, \begin{pmatrix}u & 0\\ 0 & 1 \end{pmatrix} ~\bigg|~ c \in \C, \; u \in U_1 \right\rangle.
\]
Define a non-Cartan action of $G$ on $S^{2k-1}$ by epimorphism to $U_1$ then use Lininger.
The diagonal action of $G$ on the product of $S^{2k-1}$ and homogeneous space $\C = G/U_1$ is Palais, as well as free.
The orbit space $(S^{2k-1} \times \C)/G = S^{2k-1}/U_1$ is not a topological manifold, though the projection from $S^{2k-1} \times \C$ is a principal $G$-bundle.
In particular, none in this uncountable family of free $G$-actions can be locally linear.

The same holds for $G = U_1 \x G'$ with $G'$ a linear Lie group and $M=S^{2k-1} \x G'$.
\end{exm}

We end with a family of examples whose linear Lie group $G$ is \emph{arbitrarily large}.

\begin{exm}[Lininger]
For each $n>k+1 \geqslant 3$, there are uncountably many inequivalent semifree $SO_k$-actions on $S^n$ whose fixed set is a wild $(n-k-1)$-sphere \cite[11]{Lininger2}.
Again, the construction arises from the quotient by any wild arc \cite{AC}.

The Lie group $G = \Isom^+(\R^k) = \R^k \rtimes SO_k$ is a closed subgroup of $GL_{2k}(\R)$:
\[
\left\langle \begin{pmatrix}1 & t\\ 0 & 1\end{pmatrix}, \begin{pmatrix}r & 0\\ 0 & 1 \end{pmatrix} ~\bigg|~ t \in \R^k, \; r \in SO_k \right\rangle.
\]
Define a non-Cartan action of $G$ on $S^n$ by epimorphism to $SO_k$ then use Lininger.
The diagonal action of $G$ on the product of $S^n$ and homogeneous space $\R^k = G/SO_k$ is Palais.
The orbit space $(S^n \times \R^k)/G = S^n/SO_k$ minus the singular set is not a manifold, so none in this uncountable family of semifree $G$-actions is locally linear.
\end{exm}

\subsection*{Acknowledgements}
I am grateful to Christopher Connell for various basic discussions and to Ian Hambleton for positing the avoidance of handlebody structures.

\bibliographystyle{alpha}
\bibliography{CountableApproximation_linearG}

\end{document}